\documentclass[a4paper,11pt]{article}
\usepackage[margin=3cm]{geometry}

\usepackage[T1]{fontenc}
\usepackage[utf8]{inputenc}

\usepackage{amssymb,amsmath,amsthm}
\usepackage{enumitem}
\usepackage{hyperref}
\usepackage{cleveref}
\usepackage{caption}

\usepackage{mathtools}  % pour les \coloneqq :=
\usepackage{float}

\usepackage[draft,inline,nomargin,index]{fixme} % pour mettre plein de commentaires partout !
\fxsetup{theme=color,mode=multiuser}			% en couleur, avec ça, et à plusieurs !
\FXRegisterAuthor{flm}{aflm}{\color{blue}FLM}	% \flmnote{Note de FLM}
\FXRegisterAuthor{fh}{afh}{\color{red}FH}   % \afhnote{Note de FH}

\title{On the Rokhlin lemma for infinite measure-preserving bijections}
\author{Fabien Hoareau and François Le Maître}

\newcommand{\inv}{^{-1}}

\renewcommand{\leq}{\leqslant}
\renewcommand{\geq}{\geqslant}
\renewcommand{\epsilon}{\varepsilon}

\newcommand{\Z}{\mathbb{Z}}

\newcommand{\R}{\mathbb{R}}
\newcommand{\N}{\mathbb{N}}

\DeclareMathOperator{\Aut}{\mathrm{Aut}}

\DeclareMathOperator{\supp}{\mathrm{supp}}

\newtheorem{theorem}{Theorem}[section]

\newtheorem{lemma}[theorem]{Lemma}
\newtheorem{proposition}[theorem]{Proposition}

\theoremstyle{definition}

\newtheorem{definition}[theorem]{Definition}

\newtheorem{remark}[theorem]{Remark}

\newtheorem*{question}{Question}

\begin{document}

\maketitle

\begin{abstract}
  We study the Rokhlin lemma in the context of infinite measure-preserving bijections,
  and completely classify such bijections up to 
   $\lambda$-approximate 
  conjugacy, where $\lambda$ is the infinite measure which is preserved. 
  This sharpens considerably
  the classical version of the Rokhlin lemma, which provides such a classification 
  up to $\mu$-approximate conjugacy where $\mu$ is a probability measure equivalent
  to $\lambda$. 
\end{abstract}

{
		\small	
		\noindent\textbf{{MSC-classification:}}	
		 	 37A40. 
}

\tableofcontents
\section{Introduction}

In the probability measure-preserving (pmp) context, Rokhlin's lemma states that given
any aperiodic pmp bijection $T$ of a standard probability space $(X,\mu)$, for any $N\geq 1$ and $\epsilon>0$ 
there is a  Borel subset $A\subseteq X$ such that $A,T(A),\dots,T^{N-1}(A)$ are pairwise disjoint and 
\begin{equation}\label{eq: basic Rokhlin pmp}
\mu(X\setminus(A\sqcup T(A)\sqcup\cdots\sqcup T^{N-1}(A)))<\epsilon.
\end{equation}
It has the following fundamental consequence, often also called the Rokhlin lemma: 
all aperiodic pmp bijections are 
$\mu$-approximately conjugate, namely for all $\epsilon>0$ and all aperiodic 
pmp bijections $T_1$ and $T_2$, we can find another pmp bijection $S$ such that 
\begin{equation}\label{eq: almost conjugacy}
\mu(\{x\in X\colon ST_1S\inv(x)\neq T_2(x)\})<\epsilon. 
\end{equation}

In this note, our aim is to investigate versions of  the Rokhlin lemma in 
the infinite measure-preserving setup, but where we keep working with the infinite 
measure rather than replacing it with a finite measure equivalent to it.

So let us fix a standard $\sigma$-finite infinite measured space $(X,\lambda)$ and an aperiodic 
measure-preserving bijection $T$ of $(X,\lambda)$. 
Let us first point out that the most basic version of Rokhlin's lemma
(namely that given by Equation~\eqref{eq: basic Rokhlin pmp}) holds if one replaces 
$\lambda$ by $\mu$, for $\mu$ a probability measure equivalent to $\lambda$ 
(see Proposition~\ref{prop: basic Rokhlin} for details), which 
yields that all aperiodic infinite measure-preserving bijections are 
$\mu$-approximately conjugate (see Proposition~\ref{prop: mu approximate conjugacy}).

However, it is not true that all aperiodic infinite measure-preserving
bijections are $\lambda$-approximately conjugate in the following sense.

\begin{definition}
	Two measure-preserving bijections $T_1$ and 
	$T_2$ of a standard infinite space $(X,\lambda)$ are 
	\textbf{$\lambda$-approximately conjugate} 
	if for all $\epsilon>0$
	we can find another measure-preserving bijection $S$ such that 
	\begin{equation}\label{eq: lambda almost conjugacy}
	\lambda(\{x\in X\colon ST_1S\inv(x)\neq T_2(x)\})<\epsilon. 
	\end{equation}
\end{definition}

The main reason why aperiodic infinite measure-preserving bijections can fail to be
$\lambda$-approximately conjugate is 
that this notion detects \emph{dissipativity}. Recall
that $T$ is dissipative when it is aperiodic and it admits a fundamental domain,
namely a Borel subset $D$ such that $D$ intersects almost every $T$-orbit exactly once.
Up to conjugacy by measure-preserving bijections, dissipative bijections 
are uniquely determined 
by the measure of their fundamental domains, and a model for 
a dissipative bijection whose fundamental domains have measure $d$ is given by
$x\mapsto x+d$ on $X=\R$ endowed with the Lebesgue measure. 
Our first observation is that any 
two dissipative bijections are $\lambda$-approximately conjugate if and only if
their fundamental domains have the same measure. 

More generally, Hopf's decomposition associates to every infinite measure-preserving bijection $T$ 
a unique (up to a null set) maximal $T$-invariant set $\mathfrak D_T$ 
such that the restriction of $T$ to $\mathfrak D_T$ is dissipative, and then 
the restriction of $T$ to $\mathfrak C_T$ is conservative 
(see Definition~\ref{def: conservativity} for the intrinsic
definition of conservativity).
We can now state our main result, 
which is a classification of aperiodic infinite measure-preserving 
bijections up to $\lambda$-approximate conjugacy. 

\begin{theorem}\label{thm: approx conj complete}
	Let $T_1,T_2 \in \Aut(X,\lambda)$ be aperiodic. 
	Then $T_1$ and $T_2$ are $\lambda$-approximately conjugate if and only if their dissipative parts 
	admit fundamental domains of equal measure.
	In particular if both $T_1$ and $T_2$ are conservative, then they are $\lambda$-approximately conjugate.
\end{theorem}
\begin{remark}
Our theorem can be compared with the following result of Choksi and Kakutani 
\cite[Prop.~4]{choksiResidualityErgodicMeasurable1979}: 
for every \emph{aperiodic} bijection $T_1$ and $\epsilon>0$, there exists 
an ergodic (hence conservative) bijection $T_2$ such that 
$\lambda(\{x\in X\colon T_1(x)\neq T_2(x)\})<\epsilon$. 
In their result, the conjugacy class of $T_2$ very much depends on $\epsilon$, 
while in ours it is fixed beforehand.

One can also view our result as a refinement of the work of Alpern, Choksi and Prasad \cite{alpernConjugatesInfiniteMeasure1988},
where it is proven that any two
ergodic (hence conservative) infinite measure-preserving bijections become conjugate when restricted 
to appropriate coinfinite measurable subsets. 
Indeed we are able to restrict to much larger subsets as they are cofinite of arbitrarily small complement.
\end{remark}

The key ingredient of our result in the conservative case is that given any two aperiodic 
conservative transformations $T_1$ and $T_2$, 
there are two small subsets with respect to which $T_1$ and $T_2$ have the 
exact same return time distribution (see Theorem \ref{thm: return time distribution}). 
One can then match the Kakutani-Rokhlin towers associated
to $T_1$ and $T_2$ to construct the desired conjugation. 
When the dissipative parts are not 
trivial, we are able to ``absorb'' the orbits of the conservative part into those of the dissipative part, see Lemma \ref{lem: dissipative part absorbs conservative part}.

Going back to the conservative case and Theorem \ref{thm: return time distribution},
let us highlight that one cannot just specify a return time distribution and hope that
it shows up in any conservative aperiodic bijection. For instance,
a return time distribution with only even return times will force the even levels
of the Kakutani-Rokhlin towers to be $T^2$-invariant, so that if $T$ does have such a return time 
distribution, then $T^2$ is not ergodic. Nevertheless, it would be interesting 
to understand whether there is some return time distribution which shows up in any 
infinite measure-preserving conservative aperiodic bijection.

\begin{remark}
	As pointed out to us by Emmanuel Roy, 
	there exist infinite measure-preserving ergodic bijections
	whose return time distributions on finite measure sets always have infinite Shannon entropy
	(explicit examples can be found in \cite[Prop.~2.1]{aaronsonPredictabilityEntropyInformation2009} and	
	\cite{janvresseZeroKrengelEntropy2012}).
	In particular, if there is a universal return time distribution (namely, 
	a return time distribution that shows up in any conservative aperiodic transformation),
	then it must have infinite Shannon entropy. 
\end{remark}

Our paper is organized as follows: after presenting the setup 
and proving a version of Rokhlin's lemma in the preliminary Section~\ref{sec: prelim}, we prove Theorem~\ref{thm: approx conj complete} in 
Section~\ref{sec: proof of main}. We finally comment 
on possible extensions of our result to infinite measure-preserving actions
of amenable groups in Section~\ref{sec: questions amenable}.

\section{Preliminaries}\label{sec: prelim}

Let us fix once and for all a \emph{standard $\sigma$-finite infinite measure space} $(X,\lambda)$, namely 
a standard Borel space $X$ endowed with a $\sigma$-finite infinite atomless measure $\lambda$. 
Such a measured space is isomorphic to the real line endowed with the Lebesgue measure. 
We are interested in the group $\Aut(X,\lambda)$ of all measure-preserving bijections 
$T:X\to X$. As usual, we neglect what happens on null sets.

\begin{definition}\label{def: conservativity}

	A measure-preserving bijection $T\in\Aut(X,\lambda)$ is called 
	
	\begin{itemize}
	\item \textbf{aperiodic} when all its orbits are infinite;
	\item \textbf{dissipative} when there is a Borel subset $D\subseteq X$ such that 
	\[
	X=\bigsqcup_{n\in\Z}T^n(D),
	\]
	this is then equivalent to $T$ being aperiodic and admitting a Borel fundamental domain,
	namely a subset $D$ intersecting each $T$-orbit exactly once;
	\item \textbf{conservative} if for all Borel $A\subseteq X$ of positive measure, 
	every $x\in A$ eventually returns to $A$ via $T$, or in a formula: 
	\[
	A\subseteq \bigcup_{n\geq 1}T^{-n}(A).
	\]
	
	\end{itemize}
	
\end{definition}
Recall that the Hopf decomposition associates to any $T\in\Aut(X,\lambda)$ 
a unique $T$-invariant partition 
$X=\mathfrak C_T\sqcup \mathfrak D_T$ such that the restriction of $T$ to $\mathfrak C_T$
is conservative, while the restriction of $T$ to $\mathfrak D_T$ is dissipative. 
These respective restrictions are called the \textbf{conservative part} and the \textbf{dissipative part}
of $T$. 
Observe that all fundamental domains of the dissipative part of $T$
must have the same measure since one can be obtained from the other 
by translating piecewise by powers of $T$.

Finally, recall that the \textbf{support} of $T\in\Aut(X,\lambda)$ 
is the Borel set of all $x\in X$ such that $T(x)\neq x$. 

\begin{lemma}
	{\label{lem: small set meeting all orbits}}
	Let $(X,\lambda)$ be a standard infinite measure space, and $T \in \Aut(X,\lambda)$ be conservative and aperiodic on its support. 
	For any $\varepsilon > 0$ there exists a Borel subset $B$ of $\supp T$ 
	such that $B$ intersects all nontrivial $T$-orbits, and $\lambda( B ) < \varepsilon$.
\end{lemma}

\begin{proof}
	Let $\varepsilon >0$. Since $\lambda$ is atomless and $\sigma$-finite we can write 
	$\supp T$ as 
	$ \supp T = \bigsqcup_{n \in \mathbb{N}} X_n$, with $\lambda(X_n)$ finite for all $n\in \mathbb{N}$. 
	By conservativity, for $\lambda$-almost all $x \in X_n$ the $T$-orbit 
	of $x$ intersects $X_n$ infinitely many times, in other words 
	after throwing away a null set, the restriction $\mathcal R_n$
	of the equivalence relation $\mathcal R_T=\{(x,T^kx)\colon x\in X, k\in\Z\}$ to $X_n$ is aperiodic. 
	By the marker lemma~\cite[Lem.~6.7]{KechrisMiller2004}, 
	for every $n$ there is a decreasing sequence $(B_{n,k})_k$ of Borel 
	subsets of $X_n$ which intersects every $\mathcal R_n$-class and 
	satisfies $\bigcap_k B_{n,k}=\emptyset$. In particular, since $\lambda(X_n)$ is finite, 
	for some large enough $k_n$ we have $\lambda(B_{n,k_n})<\epsilon 2^{-n-1}$. 
	It follows that $B\coloneqq \bigsqcup_n B_{n,k_n}$ is as desired.
\end{proof}

Given a conservative $T\in\Aut(X,\lambda)$ and $A\subseteq X$ Borel, we define the
return time function $\tau_{A,T}$ by: for all $x\in A$, 
\[
\tau_{A,T}(x)=\min\{n\geq 1\colon T^n(x)\in A\}.
\]
We then have that the $T$-saturation of $A$ decomposes as a \textbf{Kakutani-Rokhlin partition} as follows: 
\begin{equation}\label{eq: KR partition}
	\bigcup_{n\in\Z}T^n(A)=\bigsqcup_{n=1}^\infty\bigsqcup_{k=0}^{n-1}T^k(\tau_{A,T}\inv(n))
\end{equation}
For every $n\geq 1$, we say that $\bigsqcup_{k=0}^{n-1}T^k(\tau_{A,T}\inv(n))$ is 
a tower of height $n$, and that $T^k(\tau_{A,T}\inv(n))$ is the level $k$ of such a tower. 
Lemma~\ref{lem: small set meeting all orbits} and the notion of Kakutani-Rokhlin partition are all we need in order to prove a natural version of 
Rokhlin's lemma in the infinite measure-preserving setup.
The proof is a straightforward adaptation the probability measure case (as in \cite[Thm.~15.4]{glasnerErgodicTheoryJoinings2003}).

\begin{proposition}\label{prop: basic Rokhlin}
Given
any aperiodic measure-preserving bijection $T$ of a standard infinite measure space 
$(X,\lambda)$, for any $N\geq 1$ and $\epsilon>0$ 
there is a  Borel subset $A\subseteq X$ such that $A,T(A),\dots,T^{N-1}(A)$ are pairwise disjoint and 
\begin{equation}\label{eq: basic Rokhlin}
\lambda(X\setminus(A\sqcup T(A)\sqcup\cdots\sqcup T^{N-1}(A)))<\epsilon.
\end{equation}
\end{proposition}
\begin{proof}
	On the dissipative part of $T$, if we let $D$ be a fundamental domain,
	we simply define the $\mathfrak D$-part of $A$ as 
	$A_{\mathfrak D}=\bigsqcup_{k\in\Z}T^{Nk}(D)$. 
	On the conservative part, take $\epsilon'=\frac{\epsilon} N$ and 
	let $B$ intersect every $T_{\restriction \mathfrak C_T}$-orbit and have measure 
	less than $\epsilon'$, as provided by Lemma~\ref{lem: small set meeting all orbits}. 
	We then have the Kakutani-Rokhlin partition 
	\[
	\mathfrak C_T=\bigsqcup_{n = 1}^{\infty} \bigsqcup_{k=0}^{n-1}T^k(\tau_{B,T}\inv(n)).
	\]
	It then suffices to take for $A_{\mathfrak C}$ the elements of the towers whose level
	is a multiple of $N$, and which still have $N-1$ levels above them, namely
	\[
	A_{\mathfrak C}=\bigsqcup_{n = N}^{\infty} \bigsqcup_{q=0}^{\lfloor\frac n N\rfloor-1}T^{qn}(\tau_{B,T}\inv(n))
	\]
	and let $A=A_{\mathfrak D}\sqcup A_{\mathfrak C}$, noting that 
	$X\setminus(A\sqcup T(A)\sqcup\cdots\sqcup T^{N-1}(A)))$ is contained in 
	$\bigcup_{k=1}^{N}T^{-k}(B)$ and hence has measure at most $N\cdot \frac \epsilon N=\epsilon$.
\end{proof}

\begin{remark}
	The above proposition is also true if one only require that $T$
	is non-singular: in this case one has to modify the above procedure 
	on the conservative part by finding $B$ intersecting all $T$-orbits such 
	that $\mu(T\inv(B)\cup\cdots\cup T^{-N}(B))<\epsilon$ 
	(see \cite[Thm.~1.5.9]{aaronsonIntroductionInfiniteErgodic1997} for the ergodic case).
\end{remark}

We finish this section by recalling how the above result yields $\mu$-approximate conjugacy 
of all aperiodic infinite measure-preserving bijections.

\begin{proposition}\label{prop: mu approximate conjugacy}
	Given any two aperiodic measure-preserving bijections $T_1$, $T_2$ of
	a standard infinite measure space $(X,\lambda)$ and a probability measure 
	$\mu$ equivalent to $\lambda$, for any $\epsilon>0$ there exists 
	$S\in \Aut(X,\lambda)$ such that 
	\begin{equation*}
	\mu(\{x\in X\colon ST_1S\inv(x)\neq T_2(x)\})<\epsilon. 
	\end{equation*}
\end{proposition}
\begin{proof}
	By absolute continuity there is $\epsilon'>0$ such that for all 
	$C\subseteq X$, $\lambda(C)<\epsilon'$ implies $\mu(C)<\epsilon/2$. 
	Let us fix $N$ such that $\frac 1N<\epsilon/2$. 
	Find $A_1$ as in Proposition~\ref{prop: basic Rokhlin} for $T_1$, 
	up to shrinking $A_1$ a bit we have the following equality:
	\[
	\lambda(X\setminus(A_1\sqcup T_1(A_1)\sqcup\cdots\sqcup T_1^{N-1}(A_1)))=\epsilon'.
	\]
	Notice that since $T$ preserves $\lambda$, 
	the above equality still holds
	if we replace $A_1$ by $T_1^k(A_1)$ for some $k\in\Z$.
	Moreover, since the first $N$ translates of $T_1^{N-1}(A_1)$ by $T_1$ are disjoint, 
	one of them has $\mu$-measure $<\frac 1N$. So up to replacing $A_1$
	by some well-chosen $T_1^{k}(A_1)$ where $0\leq k\leq N-1$, we may further assume $\mu(T_1^{N-1}(A_1))\leq\frac 1N$.

	Similarly, we can find $A_2\subseteq X$ such that $\mu(T_2^{N-1}(A_2))\leq\frac 1N$ and 
	\[
	\lambda(X\setminus(A_2\sqcup T_2(A_2)\sqcup\cdots\sqcup T_2^{N-1}(A_2)))=\epsilon'.
	\]
	%we can assume $\mu(T_2^{N-1}(A_2))\leq\frac 1N$. \fhnote{Je trouve cette réduction pas facile à suivre sans faire un dessin, mais j'ai pas vraiment trouvé de meilleure façon de le dire. Là même si dans les faits tout se passe bien, comme la décomposition qu'on a sous les yeux c'est $A_1\sqcup T_1(A_1)\sqcup\cdots\sqcup T_1^{N-1}(A_1)$, c'est un peu perturbant parce qu'on peut avoir l'impression qu'on revient en bas sans contrôle et que ça devient tout pourri...} 

	We then observe that both $A_1$ and $A_2$ must have infinite $\lambda$-measure,
	so as $(A_1, \lambda_{\restriction A_1})$ and $(A_2, \lambda_{\restriction A_2})$ are both standard $\sigma$-finite spaces, there is a measure-preserving bijection $\varphi: A_1\to A_2$. 
	We extend it to 
	\[
	\psi : A_1\sqcup T_1(A_1)\sqcup\cdots\sqcup T_1^{N-1}(A_1)\to A_2\sqcup T_2(A_2)\sqcup\cdots\sqcup T_2^{N-1}(A_2)
	\] 
	by letting for, for all $k\in\{0,...,N-1\}$ and
	$x\in A_1$: 
	\[
	\psi(T_1^k(x))=T_2^k(\varphi(x)), 
	\]
	and further extend $\psi$ to some $S\in\Aut(X,\lambda)$ 
	arbitrarily (which is possible since the complement 
	of $A_1\sqcup T_1(A_1)\sqcup\cdots\sqcup T_1^{N-1}(A_1)$ has the same measure $\epsilon'$ as the 
	complement of $A_2\sqcup T_2(A_2)\sqcup\cdots\sqcup T_2^{N-1}(A_2)$). 
	Then we do have $S(T_1(x))=T_2(S(x))$ except when 
	$x$ belongs to the complement of $A_1\sqcup T_1(A_1)\sqcup\cdots\sqcup T_1^{N-1}(A_1)$
	(which has $\mu$-measure $<\epsilon/2$)
	or to $T_1^{N-1}(A_1)$ (which has $\mu$-measure $\leq\frac 1N<\epsilon/2$).
\end{proof}

\section{Proof of the main result}\label{sec: proof of main}

The proof of Theorem \ref{thm: approx conj complete} 
is divided into three parts. In Section \ref{sec:return times},
we first show that any two aperiodic conservative 
infinite measure-preserving
bijections share a common return time distribution. 
Section \ref{sec: dissipative absorption} is devoted 
to our dissipative absorption lemma, implying
that any $T\in\Aut(X,\lambda)$
with nonempty dissipative part
is $\lambda$-approximately 
conjugate to its dissipative part. 
These two results are finally put together in
Section \ref{sec: proof of main for real this time}.

\subsection{Finding the right return times}\label{sec:return times}

Here is the key lemma, which allows us to modify the return time distribution
of any $B\subseteq X$ of finite measure intersecting all 
orbits.

\begin{lemma}
	{\label{lem: return time expansion at level n}}
	Let $T\in\Aut(X,\lambda)$ be conservative and aperiodic. 
	 Let $B\subseteq X$ such that $\lambda(B)<+\infty$ and $B$ intersects all $T$-orbits. 
	 Fix $n\in\N^\ast$, and a sequence $(\lambda_k)_{k \geqslant n}$ such that for all 
	 $k\geq n$, $\lambda_k > \lambda(\tau_{B,T}\inv(k))$. 
	 Then there exists $C\subseteq X$ disjoint from $B$ such that 
	 \begin{enumerate}[label=$\mathrm{(\roman*)}$]
		\item \label{cond: no change of prior return}For all $k<n$, we have $\tau_{B\sqcup C, T}\inv (k)=\tau_{B,T}\inv(k)$.
		\item \label{cond: change of return time n}$\lambda(\tau_{B\sqcup C, T}\inv (n))=\lambda_n$.
		\item \label{cond: not too much change after n}For all $k>n$, we still have $\lambda(\tau_{B\sqcup C,T}\inv(k))<\lambda_k$.
	 \end{enumerate}
\end{lemma}
\begin{proof}
	Let us begin by observing that since $B$ intersects almost all $T$-orbits,
	the Kakutani-Rokhlin decomposition of the infinite measure space $X$ as 
	\[
	X=\bigsqcup_{N\geq 1}\bigsqcup_{k=0}^{N-1}T^k(\tau_{B,T}\inv(N))
	\]
	yields that 
	\[
	\sum_{N\geq 1} N\lambda(\tau_{B,T}\inv(N))=\lambda(X)=+\infty.
	\]
	Note that since $B$ has finite measure, all the terms in the above sum are finite. Let 
	\[\delta=\min \left\{ \lambda_{k}-\lambda(\tau_{B,T}\inv(k)) \; \middle| \; k\in\{n+1,\dots,2n\} \right\} >0.\]
	Fix $K\in\N$ such that 
\begin{equation}\label{eq: K large enough}	
	K\delta> \lambda_n-\lambda(\tau_{B,T}\inv(n)).
\end{equation}
	Then every tower of height $k\geq Kn+n+1$ can be decomposed as
	$\left\lfloor\frac{k-n-1} n\right\rfloor\geq K$ subtowers of height $n$ stacked onto each other
	plus one final top subtower whose height belongs to $\{n+1,\dots,2n\}$. 
	The plan is now to add the top of each of these height $n$ subtowers to 
	our set $B$ until we reach the desired return time distribution (the added elements correspond to the hatched set $C$ in Figure~\ref{fig: Approximate Conjugacy Lemma}). 
	
    \begin{figure}[!h]
	\centering
	\captionsetup{margin=1cm}
	\includegraphics[scale=0.57]{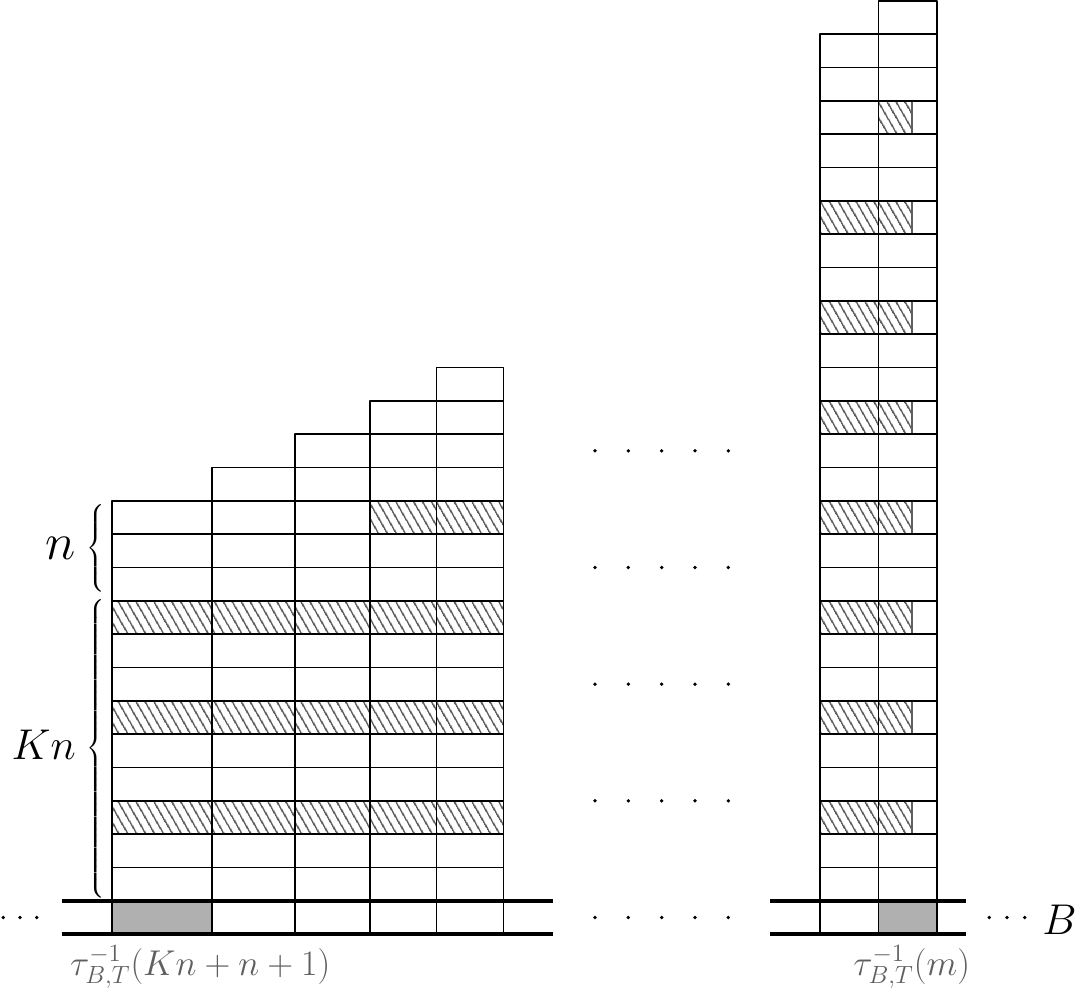}
	\caption{Illustration for $n=3$, where the hatched part represents $C$
	and $T$ translates every rectangle to the rectangle on top of it (except for the top rectangles, 
	which are only known to be taken back into $B$).}
    \label{fig: Approximate Conjugacy Lemma}
	\end{figure}
	
	Let us carry out this construction in details.
	Let $m$ be the first integer such that 
	\[
	\sum_{k=Kn+n+1}^m \left\lfloor \frac{k-n-1} n\right\rfloor \lambda(\tau_{B,T}\inv(k))\geq 
	\lambda_n-\lambda(\tau_{B,T}\inv(n)), 
	\]
	whose existence is guaranteed by the fact that 
	$\sum_{N\geq 1} N\lambda(\tau_{B,T}\inv(N))=+\infty$.
	For every $k\in\{Kn+n+1,\dots,m-1\}$ (which might be empty),
	we let \[
	C_k=\bigsqcup_{l=1}^{\left\lfloor \frac{k-n-1} n\right\rfloor}T^{nl}(\tau_{B,T}\inv(k)),
	\]
	which has measure \( \left\lfloor \frac{k-n-1} n\right\rfloor \lambda(\tau_{B,T}\inv(k))\),
	and we finally pick a subset $D\subseteq \tau_{B,T}\inv(m)$ satisfying
	\begin{equation}\label{eq: measure of D}
	\lambda(D) = \frac 1{\left\lfloor \frac{m-n-1} n\right\rfloor}
	\left(\lambda_n-\lambda(\tau_{B,T}\inv(n))-\displaystyle\sum_{k=Kn+n+1}^{m-1} \lambda(C_k)\right),
	\end{equation}
	whose existence is guaranteed by the choice of $m$, and let 
	\[ 
	C_m\coloneqq \bigsqcup_{l=1}^{\left\lfloor \frac{m-n-1} n\right\rfloor}T^{nl}(D).
	\]
	We then let \(C=\bigsqcup_{k=Kn+n+1}^m C_k\), as in Figure~\ref{fig: Approximate Conjugacy Lemma}. Since for every $k\geq Kn+n+1$ we have 
	$\left\lfloor \frac{k-n-1} n\right\rfloor\geq K$, the definition of $C_k$ yields
	that for all $k\in\{Kn+n+1,\dots,m-1\}$.
	\begin{equation}\label{eq: Ck is big}
	 K\lambda(\tau_{B,T}\inv(k))\leq \lambda(C_k).
	\end{equation}
	
	Similarly to Equation~\eqref{eq: Ck is big}, for $k=m$ we have 
	\begin{equation}\label{eq: Cm is big}
	K \lambda(D) \leq \lambda(C_m).
	\end{equation}

	Let us now understand how the return times are affected 
	when adding \(C\) to the set \(B\). 
	Observe that \(C\) is contained in the 
	Kakutani-Rokhlin towers over \(B\) whose heights belong to 
	\(\{Kn+n+1,\dots,m\}\), 
	in particular the return times of points belonging to the 
	bases of other towers are unaffected.

	Fix some \(k\in\{Kn+n+1,\dots,m-1\}\), let us compute 
	the return times of the elements of $B\sqcup C$ in the tower of height \(k\) over \(B\):
	\begin{enumerate}[label=(\alph*)]
	    \item \label{item: return time n}For each \(l\in\{0,\dots, \left\lfloor \frac{k-n-1} n\right\rfloor-1\}\), 
	 every \(x\in T^{nl}(\tau_{B,T}\inv(k))\subseteq C_k\sqcup B\), the
	\(B\sqcup C\)-return time of \(x\) is equal to \(n\), and
	    \item \label{item: return time autre}if \(x\in T^{n\left\lfloor \frac{k-n-1} n\right\rfloor}(\tau_{B,T}\inv(k))\subseteq C_k \) then the 
	\(B\sqcup C\)-return time of \(x\) is equal to 
	\[
	k-n\left\lfloor \frac{k-n-1} n\right\rfloor\in\{n+1,\dots,2n\}.
	\]
	\end{enumerate}
	For \(k=m\) we have similarly that 
	\begin{enumerate}[label=(\alph*')]
	    \item \label{item: return time n'}for each \(l\in\{0,\dots, \left\lfloor \frac{m-n-1} n\right\rfloor-1\}\), 
	and every \(x\in T^{nl}(D)\subseteq C_m \sqcup B\), the
	\(B\sqcup C\)-return time of \(x\) is equal to \(n\), while
	    \item \label{item: return time autre'} if \(x\in T^{n\left\lfloor \frac{m-n-1} n\right\rfloor}
	(D)\subseteq C_m\) its \(B\sqcup C\)-return time is equal to 
	\[m-n\left\lfloor \frac{m-n-1} n\right\rfloor\in\{n+1,\dots,2n\}.\]
	\end{enumerate}

	It follows from the above description of the new return times that 
	for all \(k<n\), we have \(\tau_{B\sqcup C,T}\inv(k)=\tau_{B,T}\inv(k)\), which 
	establishes Condition~\ref{cond: no change of prior return}.
	Moreover, it follows from \ref{item: return time n} and \ref{item: return time n'} that
	\[
	\tau_{B\sqcup C,T}\inv(n)=\tau_{B,T}\inv(n)\sqcup
	\left(\bigsqcup_{k=Kn+n+1}^{m}
	T^{-n}(C_k)\right),
	\]
	so that $\lambda(\tau_{B\sqcup C,T}\inv(n))=\lambda(\tau_{B,T}\inv(n))+\lambda(C)$.
	By Equation~\eqref{eq: measure of D} and the definition of $m$, 
	the last set $C_m$ used to change part of the height $m$ tower
	was chosen so as to obtain 
	\begin{equation}\label{eq: well chosen C}
	\lambda(C)=\lambda_n-\lambda(\tau_{B,T}\inv(n)),
	\end{equation}
	Condition~\ref{cond: change of return time n} 
	thus holds.
	Finally for all $k\in\{n+1,\dots,2n\}$, 
	the new elements with return time equal to $k$ must come from \ref{item: return time autre} or \ref{item: return time autre'}, 
	hence we have
	\[
	\tau_{B\sqcup C,T}\inv(k)\subseteq \tau_{B,T}\inv(k)\sqcup 
	\left(\bigsqcup_{l=Kn+n+1}^{m-1}T^{n\left\lfloor 
	\frac{l-n-1} n\right\rfloor}(\tau_{B,T}\inv(l))
	\right)
	\sqcup 
	T^{n\left\lfloor \frac{m-n-1} n\right\rfloor}(D).
	\]
	We thus have
	\begin{align*}
	\lambda(\tau\inv_{B\sqcup C,T}(k))-\lambda(\tau\inv_{B,T}(k))
	 & \leq \left(\sum_{l=Kn+n+1}^{m-1}\lambda(\tau_{B,T}\inv(k))\right) + \lambda(D)& \\
	& \leq \left(\sum_{l=Kn+n+1}^{m-1}\frac{\lambda(C_k)}{K} \right)+ \frac{\lambda(C_m)}{K} &(\text{by Equations }\eqref{eq: Ck is big}\text{ and }\eqref{eq: Cm is big}).
	\end{align*}
	Moreover, since $\displaystyle C=\bigsqcup_{k=Kn+n+1}^m C_k$, we have 
	$\displaystyle\lambda(C)=\sum_{k=Kn+n+1}^{m}\lambda(C_k)$, and hence 
	\[
	 \lambda(\tau\inv_{B\sqcup C,T}(k))-\lambda(\tau\inv_{B,T}(k))\leq \frac{\lambda(C)}K  =\frac{\lambda_n-\lambda(\tau\inv_{B,T} (n))}K  <\delta, 
	\]
	using first Equation~\eqref{eq: well chosen C}
	and then Inequality~\eqref{eq: K large enough}.
	It then follows from
	 the definition of $\delta$ that Condition~\ref{cond: not too much change after n}
	holds for $k\in\{n+1,\dots,2n\}$.
	To conclude, note that for all $k\geq 2n+1$ we have 
	$\tau_{B\sqcup C, T}\inv(k)\subseteq \tau\inv_{B,T}(k)$
	so that Condition~\ref{cond: not too much change after n}
	fully holds.
\end{proof}

\begin{lemma}
	{\label{lem: iterated return time expansion}}
	Let $T\in\Aut(X,\lambda)$ be conservative and aperiodic. 
	Let $A\subseteq X$ intersect almost all $T$-orbits. 
	Then for every sequence $(\lambda_n)_{n\geq 1}$ such that 
	$\lambda_n>\lambda(\tau_{A,T}\inv(n))$
	for all $n\geq 1$, there is $B$ such that $A \subseteq B$ and for all $n\in\N^\ast$,
	$\lambda(\tau_{B,T}\inv(n)) = \lambda_n$. 
\end{lemma}

\begin{proof}
	Lemma~\ref{lem: return time expansion at level n} allows us to inductively build a
	sequence of disjoint subsets $(C_n)_{n\geq 1}$ such that for all $n\geq 1$, the following hold: 
	 \begin{enumerate}[label=$\mathrm{(\roman*)}$]
		\item For all $k<n$, we have $\tau_{A\sqcup C_1\sqcup\cdots\sqcup C_n, T}\inv (k)=\tau_{A\sqcup C_1\sqcup\cdots\sqcup C_{n-1},T}\inv(k)$.
		\item $\lambda(\tau_{A\sqcup C_1\sqcup\cdots\sqcup C_n, T}\inv (n))=\lambda_n$.
		\item For all $k>n$, we have $\lambda(\tau_{A\sqcup C_1\sqcup\cdots\sqcup C_n,T}\inv(k))<\lambda_k$.
	 \end{enumerate}
	 It is then a straightforward consequence of (i) and (ii) that the set 
	 $B\coloneqq A\sqcup \bigsqcup_{n\in\N}C_n$ is as desired.
\end{proof}

We can now state and prove the core result of our paper. 

\begin{theorem}\label{thm: return time distribution}
	Let $T_1,T_2 \in \Aut(X,\lambda)$ be conservative aperiodic. 
	For any $\varepsilon>0$ there exists two subsets $B_1$ and $B_2$ of $X$ meeting almost all $T_1$ and $T_2$-orbits respectively, of measure $\lambda(B_1) = \lambda(B_2) < \varepsilon$, and such that for any $n \in \N^\ast$ we have: $\lambda(\tau\inv_{B_1,T_1}(n)) = \lambda(\tau\inv_{B_2,T_2}(n))$. 
\end{theorem}

\begin{proof}
	By Lemma~\ref{lem: small set meeting all orbits}, we can 
	fix two Borel sets $A_1$ and $A_2$ of measure $<\frac{\varepsilon}{2}$, 
	meeting almost all $T_1$-orbits and $T_2$-orbits respectively. 
	We then pick a sequence $(\lambda_n)_{n\geq 1}$ such that for all $n\geq 1$, 
	$\lambda_n > \max \left\{  \lambda(\tau\inv_{A_1,T_1}(n)) , 
	\lambda(\tau\inv_{A_2,T_2}(n)) \right\}$ while $\sum_{n} \lambda_n < \varepsilon$.
	We conclude by applying Lemma~\ref{lem: iterated return time expansion} to $T_1$ and $A_1$,
	and then to $T_2$ and $A_2$, both times 
	with this sequence $(\lambda_n)$.
\end{proof}

\subsection{The dissipative absorption lemma}
\label{sec: dissipative absorption}

In order to prove Theorem \ref{thm: approx conj complete}, 
we require the following lemma, which 
allows the dissipative part to ``absorb'' the conservative part.

\begin{lemma}{\label{lem: dissipative part absorbs conservative part}}
    Let $\epsilon>0$ and $T\in\Aut(X,\lambda)$ whose conservative and dissipative parts both have positive measure, and let $D$ be a fundamental domain for the dissipative part of $T$. 
    Then there exists a dissipative $\widetilde{T} \in \Aut(X,\lambda)$
    such that $\lambda(\{ x \in X \colon T(x) \neq \widetilde{T}(x) \})< \epsilon$ and $D$ is also a fundamental domain for $\widetilde T$.
\end{lemma}
\begin{proof}
    We obtain $\widetilde T$ by patching together the 
	dissipative part and the conservative part of $T$ (see Figure \ref{fig: absorption}
	for an illustrative example).
	
	Let $D$ be a fundamental domain of the dissipative part of $T$.
	By Lemma~\ref{lem: small set meeting all orbits} we find 
	$B\subseteq \mathfrak C_T$ such that  $\lambda(B)< \frac\epsilon2$, 
	$\lambda(B)\leq\lambda(D)$,
	and $B$ intersects every orbit of the conservative part $\mathfrak C_T$ of $T$. 
	We then fix $A\subseteq D$ such that  $\lambda(A)=\lambda(B)$.
	Finally, let $\phi: A\to B$ and $\psi: T\inv(B)\to T(A)$ be measure-preserving bijections. 

	\begin{figure}[H]
		\centering
		\includegraphics[scale=0.57]{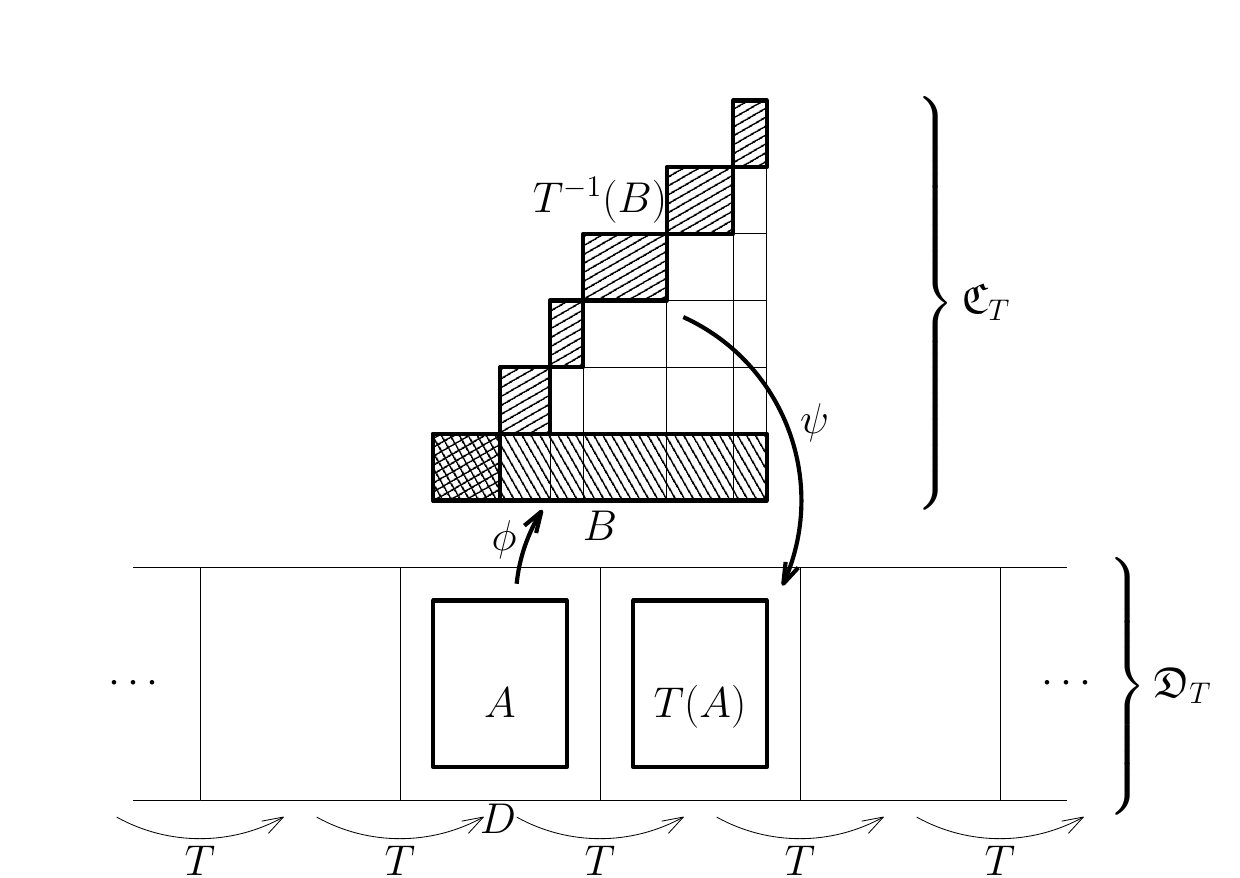}
			\captionsetup{margin=1.3cm}
		\caption{
			An example of maps $\phi,\psi$ that 
			are used to modify the dynamics of $T$. 
			The base $B$ of the Kakutani-Rokhlin partition of $\mathfrak{C}_T$ is the hatched part
			and $T\inv(B)$ is the (differently hatched) top of the towers. 
		}\label{fig: absorption}
	\end{figure}

	We define $\widetilde T$ as follows: 
	\[
	\widetilde{T} (x) = \left\lbrace \begin{array}{cl}
		T(x) &  \mbox{ if }  x \in X\setminus A\sqcup T\inv(B), \\
		\phi(x) & \mbox{ if }  x \in A, \\
		\psi(x) & \mbox{ if }   x \in T\inv(B).
	\end{array}
	\right.
	\]

	Let us check that $\widetilde T$ is dissipative with fundamental domain $D$. Since $D$ is a fundamental domain of the 
	dissipative part of $T$
	and since $\widetilde T$ and $T$ coincide on $\bigsqcup_{n\leq 0}T^n(D)$,
	all the negative $\widetilde T$-translates of $D$ are disjoint from each other. 
	Hence $D$ intersects every $\widetilde T$-orbit in at most one point,
	and $D$ intersects the $\widetilde T$-orbit of every element of $\bigsqcup_{n\leq 0}T^n(D)$.
	Denote by $Y$ the $\widetilde T$-saturation of $D$, we have to show $Y=X$, and for now we know
	that $Y\supseteq \bigsqcup_{n\leq 0}T^n(D)$.
	Since $\widetilde T$ and $T$ coincide on $\bigsqcup_{n\geq 0}T^n(D\setminus A)$, 
	we further have $Y\supseteq \bigsqcup_{n\geq 1}T^n(D\setminus A)$.
	Now because  $\widetilde T(A)=B$, we have $B\subseteq Y$.
	If $x\in\mathfrak C_T\setminus B$, let $k\geq 1$ be the first integer such that $T^{-k}(x)\in B$. Then
	since $T\inv$ and  $\widetilde T\inv$ coincide on $\mathfrak C_T\setminus B$, 
	we have $\widetilde T^{-k}(x)\in B$, which yields the inclusion $\mathfrak{C}_T \subseteq Y$. 
	Finally, by construction 
	$\bigsqcup_{n\geq 1}\widetilde T^n(T\inv B)=\bigsqcup_{n\geq 1}T^n(A)$ so 
	$Y\supseteq \bigsqcup_{n\geq 1}T^n(A)$ so $Y=X$, which concludes the proof.
\end{proof}

\subsection{Proof of Theorem \ref{thm: approx conj complete}}
\label{sec: proof of main for real this time}

Before proving the main result, we isolate one last lemma which shows that 
the condition on the measure of the dissipative fundamental domains is 
necessary for $\lambda$-approximate 
conjugacy to hold. 

\begin{lemma}
	\label{lem: condition on measure of dissipative fundamental domains}
	Let $T_1\in\Aut(X,\lambda)$, let $d_1$ be the measure of a fundamental domain of its dissipative part, 
	let $\epsilon>0$, and $T_2\in\Aut(X,\lambda)$ such that 
	$\lambda(\{x\in X\mid T_1(x)\neq  T_2(x)\})\leq \epsilon$. 
	Then $T_2$ has a dissipative fundamental domain of measure $\geq d_1-\epsilon$.
\end{lemma}
\begin{proof}
	Let $D_1$ be a fundamental domain for the dissipative part of $T_1$,
	then let $C_1$ denote the set of $x\in D_1$ whose $T_1$-orbit intersects the set 
	$A\coloneqq\{x\in X\mid T_1(x)\neq T_2(x)\}$,
	which satisfies $\lambda(A)\leq \epsilon$ by assumption. 
	Since $T_1$ is measure-preserving and $C_1$ is contained both in the fundamental domain $D_1$ 
	and in the $T_1$-saturation of $A$,
	we then have $\lambda(C_1)\leq\lambda(A)\leq \epsilon$ and thus $\lambda(D_1\setminus C_1)\geq d_1-\epsilon$.
	By construction $T_1$ and $T_2$ coincide on the $T_1$ orbit of any $x\in D_1\setminus C_1$ so $D_1\setminus C_1$ 
	is contained in a fundamental domain for 
	the dissipative part of $T_2$, which finishes the proof.
\end{proof}

\begin{proof}[Proof of Theorem~\ref{thm: approx conj complete}]
	We begin by the direct implication: assume that $T_1$ and $T_2$ are $\lambda$-approximately conjugate.
	Let $d_1$ and $d_2$ denote the respective measures of fundamental domains for the dissipative parts of $T_1$ and $T_2$.
	Then by Lemma~\ref{lem: condition on measure of dissipative fundamental domains} we have $d_1\leq d_2$, but then by symmetry $d_2\leq d_1$ so  that $d_1=d_2$. 

	Let us now show the converse. Denote by $d$ the common measure of a fundamental domain of $T_1$ and $T_2$.
	
	Let us begin with the case $d=0$.
	We apply Theorem~\ref{thm: return time distribution}
	and thus get $B_1$, $B_2$ meeting almost all $T_1$ and $T_2$
	orbits respectively, with $\lambda(B_1)=\lambda(B_2)<\epsilon$, and such that 
	for all $n\in\N^*$ we have $\lambda(\tau\inv_{B_1,T_1}(n)) = \lambda(\tau\inv_{B_2,T_2}(n))$. 
	Observe that since $B_1$ meets all $T_1$-orbits, we have a Kakutani-Rokhlin partition of $X$ given by: 
	\[
	X=\bigsqcup_{n\geq 1}\bigsqcup_{k=0}^{n-1}T_1^k(\tau_{B_1,T_1}\inv(n)),
	\]
	and similarly for $T_2$: 
	\[
	X=\bigsqcup_{n\geq 1}\bigsqcup_{k=0}^{n-1}T_2^k(\tau_{B_2,T_2}\inv(n)),
	\]
	Let us fix a measure-preserving bijection $\varphi: B_1 \to B_2$ satisfying $\varphi(\tau\inv_{B_1,T_1}(n)) = \tau\inv_{B_2,T_2}(n)$ for any $n \in \N^\ast$. 
	We extend it to a tower-preserving element $S$ of $\Aut(X,\lambda)$ 
	by letting for, for all $n\in\N^\ast$, all
	$x\in \tau_{B_1,T_1}\inv(n)$ and all $k\in\{0,\dots,n-1\}$
	\[
	S(T_1^k(x))=T_2^k(\varphi(x)). 
	\]
	By construction, we always have $ST_1(x)=T_2S(x)$, except when $x$ is on the top of some $T_1$-tower, \textit{i.e.} when 
	$\displaystyle x\in \bigsqcup_{n \geq 1} T_1^{n-1}(\tau_{B_1,T_1}\inv(n))$. So 
	
	\begin{align*}
	\lambda(\{x\in X\colon ST_1S\inv(x)\neq T_2(x)\}) & \leq 
	\sum_{n=1}^\infty \lambda(T_1^{n-1}(\tau_{B_1,T_1}\inv(n))) \\
	&= \sum_{n=1}^\infty \lambda(\tau_{B_1,T_1}\inv(n)))=\lambda(B_1)<\epsilon,
	\end{align*}
	as desired. \\

	We are now left with the case $d>0$ (it can be infinite). 
	If both $T_1$ and $T_2$ are dissipative, then they are actually conjugate, in particular they are $\lambda$-approximately conjugate.
	By transitivity of $\lambda$-approximate conjugacy, we may then assume without loss of generality
	that $T_1$ is dissipative but $T_2$ is not: its conservative part has positive measure.
	
	Let us now fix $\epsilon>0$ and apply Lemma \ref{lem: dissipative part absorbs conservative part} so as to obtain $\widetilde T_2$ dissipative which shares a fundamental domain with $T_2$
	and satisfies
	\[
	\lambda(\{x\in X\colon  T_2(x)\neq \widetilde T_2(x)\})<\epsilon.
    \]
	Since $T_1$ and $\widetilde T_2$ are both dissipative with 
	a fundamental domain of measure $d>0$, they are conjugate, so there is $S\in\Aut(X,\lambda)$ such that 
	$\lambda(\{x\in X\colon  T_2(x)\neq ST_1S\inv(x)\})<\epsilon$.
	Since this can be done for all $\epsilon>0$, 
	we conclude that $T_2$ and $T_1$ are $\lambda$-approximately conjugate
	as wanted, which concludes the proof.
\end{proof}

\section{Beyond \texorpdfstring{$\Z$}{Z}-actions}\label{sec: questions amenable}

In the probability context, given an infinite amenable group $\Gamma$, Ornstein and 
Weiss' quasi-tiling machinery yields that 
all free probability measure-preserving $\Gamma$-actions on $(X,\mu)$ are
$\mu$-approximately conjugate \cite{ornsteinEntropyisomorphismtheorems1987}. 
More generally, it is a result of Elek that given any infinite countable group $\Gamma$, 
any two hyperfinite probability measure-preserving $\Gamma$-actions
are $\mu$-approximately conjugate if and only if they share the same invariant random subgroup
\cite[Thm.~9]{elekFinitegraphsamenability2012} 
(see also \cite[Sec.~2]{giraudHyperfiniteMeasurepreservingActions2022} for a recent 
more direct proof). 

In the infinite measure-preserving world however, 
our work shows that even for $\Gamma=\Z$, the distribution of point stabilizers 
does not provide the whole information 
for $\lambda$-approximate conjugacy since dissipativity is not detected. 
Trying to extend our main result to other amenable groups, one way around this problem 
might be to require \emph{elementwise conservativity} of the action, 
as recently introduced by Glasner and Lederle \cite{glasnerStrongSubgroupRecurrence2025}.

\begin{question}
	Let $\Gamma$ be an infinite amenable group. Is it true that all elementwise 
	conservative free $\Gamma$-actions on a standard infinite measure space 
	$(X,\lambda)$ are $\lambda$-approximately conjugate, namely for any two such actions 
	$\alpha_1$, $\alpha_2$, all $F\subseteq \Gamma$ finite and all $\epsilon>0$, 
	is there $S\in\Aut(X,\lambda)$ such that for all $\gamma\in F$, 
	\[
	\lambda(\{x\in X\colon S\alpha_1(\gamma)S\inv x\neq\alpha_2(\gamma)(x)\})<\epsilon?
	\]
\end{question}
 We don't know the answer to this question even for $\Gamma=\Z^2$ 
or $\Gamma$ virtually $\Z$.

\listoffixmes  %%%% liste tous les commentaires

%%%% Bibliographie %%%%

\bibliographystyle{alphaurl}
\bibliography{bib}

\end{document}